\documentclass[english,11pt]{paper}
\usepackage{fullpage} 
\usepackage[latin9]{inputenc}
\usepackage{babel}

\usepackage{tipa}
\usepackage{amsthm}
\usepackage{amsmath}
\usepackage{graphicx}
\usepackage{amssymb}
\usepackage{wrapfig}
\usepackage[unicode=true, pdfusetitle,
 bookmarks=true,bookmarksnumbered=false,bookmarksopen=false,
 breaklinks=false,pdfborder={0 0 1},backref=false,colorlinks=false]
 {hyperref}

\makeatletter

\theoremstyle{plain}
\newtheorem{thm}{Theorem}

\theoremstyle{plain}
\newtheorem{lem}[thm]{Lemma}
\theoremstyle{plain}
\newtheorem{prop}[thm]{Proposition}
\theoremstyle{plain}

\theoremstyle{plain}
\newtheorem{cor}[thm]{Corollary}
\theoremstyle{plain}

\usepackage[nothing]{algorithm} \usepackage[noend]{algorithmic} 
\usepackage{float}
\newdimen\algorithmicindent \algorithmicindent=0.5cm

\usepackage{color}
\graphicspath{{figures/}{./}}

\DeclareMathOperator*{\argmax}{argmax}

\newif\ifnotesw\noteswtrue

\newif\ifnotesw\noteswtrue

\def\blfootnote{\xdef\@thefnmark{}\@footnotetext}

\makeatother

\begin{document}

\title{Optimal recovery of damaged infrastructure networks}

\author{Alexander Gutfraind $^{1}$, Milan Bradonji\'c $^{2}$, and Tim Novikoff $^{3}$}

\maketitle

\begin{abstract}
Natural disasters or attacks may disrupt infrastructure networks on a vast scale.
Parts of the damaged network are interdependent, making it difficult to plan and optimally execute the recovery operations.
To study how interdependencies affect the recovery schedule, 
we introduce a new discrete optimization problem where the goal is to minimize the total cost of installing (or recovering) a given network.
This cost is determined by the structure of the network and the sequence in which the nodes are installed.
Namely, the cost of installing a node is a function of the number of its neighbors that have been installed before it.
We analyze the natural case where the cost function is decreasing and convex, and provide bounds on the cost of the optimal solution. 
We also show that all sequences have the same cost when the cost function is linear
and provide an upper bound on the cost of a random solution for an Erd\H os-R\'enyi random graph.
Examining the computational complexity, we show that the problem is NP-hard when the cost function is arbitrary.
Finally, we provide a formulation as an integer program, an exact dynamic programming algorithm, and a greedy heuristic which gives high quality solutions.

\end{abstract}

\section{Introduction}
\setcounter{footnote}{3} 
\blfootnote{$^{1}$ University of Texas at Austin, 1 University Station, Austin, TX, 78712, USA, \href{mailto:agutfraind.research@gmail.com}{agutfraind.research@gmail.com}.}
\blfootnote{$^{2}$ Mathematics of Networks and Communications, Bell Laboratories, Alcatel-Lucent, 600 Mountain Avenue, Murray Hill, New Jersey 07974, USA, \href{mailto:milan@research.bell-labs.com}{milan@research.bell-labs.com}. } 
\blfootnote{$^{3}$ Center of Applied Mathematics, Cornell University, Ithaca, New York, 14853, USA, \href{mailto:tnovikoff@gmail.com}{tnovikoff@gmail.com}.}
Many vital infrastructure systems can be represented as networks, including transport, communication and power networks.
Large parts of those networks might become inoperable following natural disasters, such as the Japanese 2011 earthquake near Sendai,
to the extent that the cost of the recovery might be measured in hundreds of billions of dollars.\blfootnote{$^{4}$%
World Bank, East Asia and Pacific Economic Update 2011, \emph{The Recent Earthquake and Tsunami in Japan:
Implications for East Asia}, Washington, DC, March 21, 2011}
Like earthquakes, wars or even terrorist attacks can cause massive network disruption.

When faced with vast devastation, authorities must develop a plan or schedule for restoring the network.
A particularly difficult and costly stage in the recovery is the initial stage because
no infrastructure is available to support recovery operations.
As the recovery progresses, previously installed nodes provide resources or bases that help reduce the cost of rebuilding their neighbors
- a phenomenon that we call neighbor aid.
For simplicity of analysis, we propose that during the reconstruction process all of the nodes and edges of the original network are to be re-installed.
We assume that the cost of the edges is not determined by the schedule of the restoration.
We examine the following question: How could the recovery schedule be optimized in order to reduce the total cost?

The phenomenon of neighbor aid applies to many decision problems beyond disaster recovery.
In fact it originally appeared in educational software design when one of the authors sought the optimal order in which to teach a list of vocabulary words.
A word like \emph{neologism}, for example, would be easier for students to learn if they already knew the etymologically adjacent words \emph{neophyte} and \emph{epilogue}.
This means that the sequence with which words are taught could be optimized to make them easier to learn.

Naturally, the problem we will formulate is related to scheduling problems such as single processor
scheduling \cite{Karp61}, the linear ordering problem \cite{Mitchell96}, the quadratic assignment problem \cite{Pardalos94}
and the traveling salesman problem (TSP) \cite{schrijver2005history}.
Like TSP our problem asks for a schedule based on a graph structure, but unlike with TSP, 
the cost associated with visiting a given node could depend on \emph{all} of the nodes visited before the given node.
Our problem offers a new model for disaster recovery of networks. For other models see e.g. \cite{Guha99,Lee07,Yu03,Hentenryck10}.

\section{General Formulation and Basic Results}
\label{sec:formulation}
In this paper we introduce a description of the neighbor-aid phenomenon as a discrete optimization problem, which we call the \emph{Neighbor Aided Network Installation Problem} (NANIP).  
An instance of NANIP is specified by a network $G=(V,E)$ and a real-valued function $f: \mathbb{N}_0 \to \mathbb{R}_{+}$.
The domain of $f$ are non-negative integers up to the degree of the highest degree node in $G$. 
The goal is to find a permutation of the nodes that minimizes the total cost of the network installation.
The cost of installing node $v_t \in V$ under a permutation $\sigma$ of $V$
is given by 
$$f(r(v_t, G, \sigma))\,,$$ where $r(v_t, G, \sigma)$ is the number of nodes adjacent to $v_t$ in $G$ that appear before $v_t$ in the permutation $\sigma$. The total cost of installing $G$ according to the permutation $\sigma$ is given by
\begin{equation}
C_G(\sigma) = \sum_{t=1}^{n} f(r(v_t, G, \sigma)).\label{eq:general-NANIP}
\end{equation}
The problem is illustrated in Fig.~\ref{fig:illustration}.

NANIP could also be expressed as optimization over $n$-by-$n$ permutation matrices $\Pi$ where 
$A$ is the adjacency matrix of $G$:
\begin{equation}
\min_\Pi \sum_{t=1}^{n} f\left(\sum_{i=1}^{t-1}(\Pi^T A \Pi)_{ti}\right) \,.
\end{equation}
The inner sum computes the number of nodes adjacent to node $t$ that are installed before $t$.
Throughout the paper we use the standard notation $n := |V|$ and $m := |E|$.
We assume that $G$ is connected and undirected, unless we note otherwise.

\begin{figure}[th]
\begin{centering}
\includegraphics[width=0.4\textwidth]{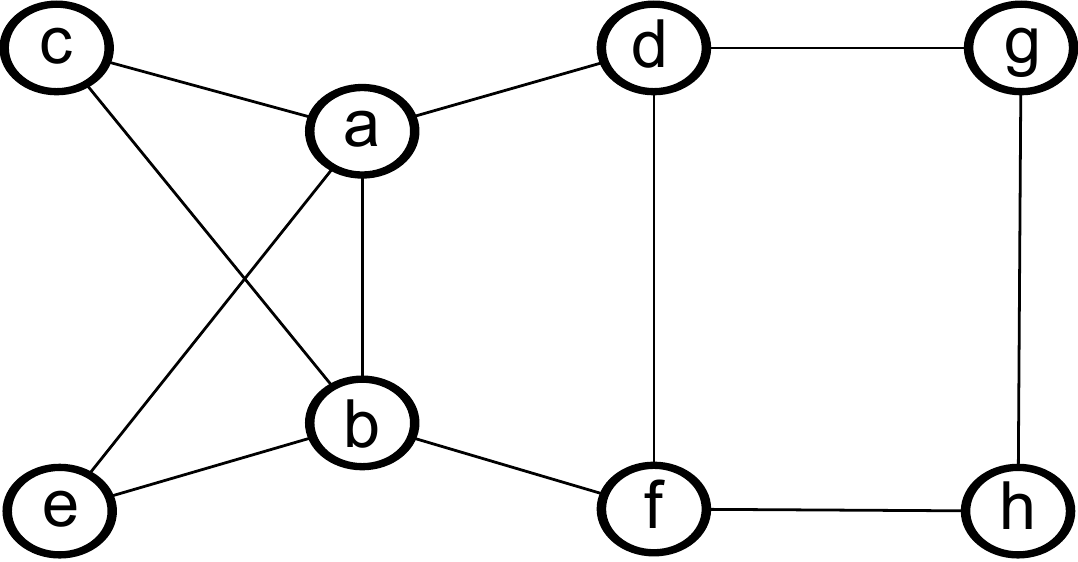} 
\par
\end{centering}
\caption{An illustration of the problem on a simple instance. 
When $f(k)=\frac{12}{1+k}$, nodes $a$ and $b$ are best installed early because they can reduce the
cost of many of their neighbors. An optimal solution is the sequence
$\sigma=(a,b,c,e,d,f,g,h)$ with cost $46=12+6+4+4+6+4+6+4$.\label{fig:illustration}}
\end{figure}

We begin with a preliminary lemma which establishes that all the arguments used in calculating the node costs must sum to $m$, the number of edges in the network.

\begin{lem}\label{lem:edge-decomp}For any network $G$, and any permutation $\sigma$ of the nodes of $G$, 
\begin{equation}
\sum_{t=1}^n r(v_t,G,\sigma) = m \label{eq:edge-decomp}\,.
\end{equation}
\end{lem}

\begin{proof}
For every edge in $G$ and any permutation $\sigma$, one of the two endpoints of the edge is the latter of the two to be installed under $\sigma$. By definition, $r(v_t,G,\sigma)$ is the number of edges whose latter endpoint is installed at time $t$. Thus summing up $r(v_t,G,\sigma)$ for all $t=1,2,\dots,n$ accounts for all the edges in $G$.
\end{proof}

One application of this lemma is the case of a linear cost function $f(k)=ak+b$, for some real numbers $a$ and $b$.  
With such a function the optimization problem is trivial in that all installation permutations have the same cost.
\begin{cor}{For any $G$, when $f(k) = ak + b$ for some real numbers $a$ and $b$, then $C_G (\sigma) = am + bn$.}
\end{cor}
\begin{proof}
\begin{equation}
\nonumber
C_G(\sigma) = \sum_{t=1}^n f(r(v_t, G, \sigma)) = \sum_{t=1}^n \Big (a r(v_t, G, \sigma) + b \Big ) = a \sum_{t=1}^n r(v_t, G, \sigma) + bn = am + bn.
\end{equation}
\end{proof}

For some applications, it may be useful to know the cost of a sequence when the sequence is chosen at random.
Our analysis of this case uses a frequently-used model of a random graph, namely the Erd\H os-R\'enyi model~\cite{Erdos60},
where any edge exists independently with probability $p$. 
\begin{lem}
\label{lm:er}
{Let $G_{n,p}$ be an Erd\H os-R\'enyi random graph on $n$ nodes with the edge probability $p$. Let $\sigma$ be a permutation of the nodes of $G_{n,p}$ chosen uniformly at random from the set of such permutations. For any real-valued function $f: \mathbb{N}_0 \to \mathbb{R}_{+}$ and $p > 0$ the expected cost satisfies
$$\mathbb{E}\Big (C_{G_{n,p}}(\sigma) \Big ) \leq \frac{1}{p}\sum_{k=0}^{n-1} f(k).$$}
\end{lem}
\noindent Notice that the graph does not need to be connected in Lemma~\ref{lm:er}.
\begin{proof}
Let $\sigma = (v_1, \ldots, v_n)$ be a permutation on $n$ nodes chosen uniformly at random from the set of all such permutations, where the label $v_i$ is assigned to the $i$\textsuperscript{th} node to appear in $\sigma$. Then when $v_t$ is installed at time $t$, the probability that exactly $k$ adjacent nodes are already installed is ${t-1 \choose k} p^k(1-p)^{t-1-k}$. Hence the expected cost of installing $v_t$ is
\begin{equation}
\nonumber
\sum_{k=0}^{t-1} {t-1 \choose k} p^k(1-p)^{t-1-k} f(k) \,,
\end{equation}
for $t = 1, \ldots, n$.
Thus, the expected total cost of installing $G$ according to $\sigma$ is 
\begin{eqnarray}
\nonumber
\mathbb{E}\Big (C_{G_{n,p}}(\sigma) \Big ) &=& \sum_{t=1}^n \sum_{k=0}^{t-1} {t-1 \choose k} p^k(1-p)^{t-1-k} f(k) = \sum_{k=0}^{n-1} \sum_{t=k+1}^n  {t-1 \choose k} p^k(1-p)^{t-1-k} f(k) \\
\nonumber
&=& \sum_{k=0}^{n-1} \sum_{q=k}^{n-1}  {q \choose k} p^k(1-p)^{q-k} f(k) = \sum_{k=0}^{n-1} \frac{p^k}{(1-p)^k } f(k) \sum_{q=k}^{n-1}  {q \choose k} (1-p)^{q}.
\end{eqnarray}
Since 
\begin{equation}
\nonumber
\sum_{t=0}^{\infty}  {t \choose k} x^{t} = \frac{x^k}{(1-x)^{k+1}} \,,
\end{equation}
for $x \in [0,1)$ and any positive integer $k$, we have that
\begin{equation}
\nonumber
\sum_{q=k}^{n-1}  {q \choose k} (1-p)^{q} \leq \sum_{q=0}^{\infty}  {q \choose k} (1-p)^{q} = \frac{(1-p)^k}{p^{k+1}} \,.
\end{equation}
Substituting this above and canceling terms we obtain the bound
\begin{equation}
\nonumber
\mathbb{E}\Big (C_{G_{n,p}}(\sigma) \Big ) \leq \frac{1}{p}\sum_{k=0}^{n-1} f(k) \,.
\end{equation}
\end{proof}
An exact expression for $\mathbb{E}\Big (C_{G_{n,p}}(\sigma) \Big )$ could be given using the $_2F_1$ hypergeometric function, where 
$$_2F_1(a,b,c,z) = \sum_{i=0}^\infty \frac{(a)_i(b)_i}{(c)_i} \, \frac {z^i} {i!}\,.$$
Using Mathematica \cite{Mathematica}, we find that the expected total cost of installing $G_{n,p}$ according to $\sigma$ is:
\begin{align*}
\mathbb{E}\Big (C_{G_{n,p}}(\sigma) \Big ) &= \sum_{k=0}^{n-1} \frac{p^k}{(1-p)^k } f(k) \sum_{q=k}^{n-1}  {q \choose k} (1-p)^{q} \\
                                   &= \sum_{k=0}^{n-1} \frac{p^k}{(1-p)^k } f(k) \left[\frac{(1-p)^k}{p^{k+1}} - (1-p)^{n}{n \choose k} {}_{2}F_{1}(1,n+1,n+1-k,1-p)\right]\\
                                   &= \frac{1}{p}\sum_{k=0}^{n-1} f(k) - \sum_{k=0}^{n-1}f(k) {n \choose k} p^k (1-p)^{n-k} {}_{2}F_{1}(1,n+1,n+1-k,1-p) \,.
\end{align*}

\section{Decreasing Convex Cost Function}\label{sec:convex}
The case of a decreasing convex cost function $f$ is natural for many applications of NANIP
because many times the neighbor aid phenomenon declines in significance as the number of neighbors increases.
For a decreasing cost function $f:\mathbb{N}_0 \to \mathbb{R}_{+}$, we say $f$ is {\it decreasing convex} if $f(i) - f(i+1) \geq f(j) - f(j+1)$ for all $j \geq i$ in $\mathbb{N}_0$.
(Notice that a differentiable, convex, non-increasing function $\bar{f}(x): \mathbb{R} \to \mathbb{R}$ satisfies $\bar{f}(x) - \bar{f}(x+1) \geq \bar{f}(y) - \bar{f}(y+1)$ for all $y \geq x$.)

It is illuminating to see the effect of convexity in $f$ for a small graph, such as a $3$-node path graph $V=\{v_1,v_2,v_3\}$ and $E=\{(v_1,v_2),(v_2,v_3)\}$.
A sequence like $(v_1,v_3,v_2)$ has cost $f(0)+f(0)+f(2)$ while a sequence like $(v_1,v_2,v_3)$ costs $f(0)+f(1)+f(1)$.
When $f$ is decreasing convex, the second sequence is at least as cheap as the first one, since $f(1) + f(1) \leq f(0)+f(2)$. 

For decreasing convex cost functions it is possible to find absolute lower bounds on the cost of installing a graph.
To state this bound we extend the domain of $f$ to non-integer values so that it is defined on all reals in $[0,n-1]$.
One approach that preserves convexity is to define $f(q)$ for every positive non-integer $q$ as the linear interpolation of $f(\lfloor q\rfloor)$ and $f(\lceil q\rceil)$.
To derive the bound, we start with the case where only a subgraph $H$ of the network needs to be installed.  
\begin{thm}
\label{thm:cost.bound}
Let $H=(V_{H},E_{H})$ be a non-empty subgraph of $G=(V_{G},E_{G})$. 
Let $\sigma$ be any installation sequence for $G$ such that it first installs $V_{G}\setminus V_{H}$ and then $V_{H}$ . 
Suppose the nodes in $V_{G}\setminus V_{H}$ have already been installed. 
Denote with $E_{GH}$ the cut from $V_{G}\setminus V_{H}$ to $H$: 
\begin{equation}
E_{GH}=\left\{ \{i,j\}\in E_{G} : i\in V_{G}\setminus V_{H}, j\in V_{H}\right\} \,.
\end{equation}

Then for any decreasing convex cost function $f$, the total cost of installing $V_H$ satisfies
\begin{equation}
C_H(\sigma)\geq|V_{H}|f\left(\frac{|E_{H}|+|E_{GH}|}{|V_{H}|}\right).\label{eq:cvx-bd}
\end{equation}
\end{thm}
The scenario in Theorem~\ref{thm:cost.bound} frequently arises in practical network recovery situations.  
The graph $G$ may represent the total network of a region while the subgraph $H$ represents just the damaged subgraph.  
As a result, the cost of rebuilding $H$ is reduced by connections to existing infrastructure (the edges $E_{GH}$).

\begin{proof}
The argument follows from Jensen's inequality
\begin{equation}
\label{eq:jensens}
f\left(\frac{\sum_{i=A}^B {x_{i}}}{B-A+1}\right) \leq \frac{\sum_{i=A}^B {f(x_{i})}}{B-A+1} \,,
\end{equation}
with $x_{i}=r(v_i, G, \sigma)$,  $A=n+1-|V_{H}|$, and $B = n$.
\end{proof}

\noindent This bound suggests that the best installation sequence is one where the cost
for every node is close to the average cost. 
A special case is where $V_{G}\setminus V_{H}$ is a single node. 
Its installation cost must be $f(0)$ and so we have a lower bound on the cost of $G$.
\begin{cor}
\label{cor.inst.seq}
Let $G=(V,E)$ be any graph and $\sigma$ any installation sequence. 
If $f$ is non-increasing convex then \label{cor:jensen}\[
C_G(\sigma)\geq f(0)+(n-1)f\left(\frac{m}{n-1}\right).\]\label{eq:cvx-bdG}
\end{cor}
    
An immediate application of this bound is to find the optimal solutions for tree graphs.
On a tree $T$, the bound gives $c(T)\geq f(0) + (n-1) f(1)$.
Consider any algorithm that picks a random starting node and then proceeds to install neighbors of installed nodes
until all the nodes are installed. (An explicit example is stated later in the paper as Algorithm~\ref{al:greedy}.)
Such an algorithm would install $T$ at a cost $f(0) + (n-1) f(1)$, which must be optimal by Corollary~\ref{cor.inst.seq}.

We now proceed to give another bound by considering the degrees of the nodes.
Assume that the nodes are labeled in order of increasing degree in $G$. 
Thus, if $d_i$ gives the degree of node $v_i$, then $d_1 \leq d_2 \leq \cdots \leq d_n$.
We consider the following relaxation of NANIP (\ref{eq:general-NANIP}), 
in which we do not seek an installation sequence from which the values of $r(v_t,G,\sigma)$ are deduced, 
but rather we seek a collection of numbers, $p_1, \ldots , p_n$, 
which satisfy constraints that are satisfied also by $r(v_t,G,\sigma)$ in NANIP.
More precisely for every $i$ we require that $p_i$ is less than or equal to $d_i$, and we require that values of $\sum_{i=1}^n p_i = m$ (compare to Lemma \ref{lem:edge-decomp}). Thus we have the following minimization problem.
\begin{align}
   P = \min & \sum_{i=1}^n f(p_i) \label{eq:relaxationP}\\
    \mbox{s.t.} & \sum_{i=1}^n p_i = m \notag\\
    & 0 \leq p_i \leq d_i \textrm{ and } p_i \in \mathbb{N}_0 \textrm{ for } i=1,2,\dots,n\notag.
\end{align}

\begin{figure}[hbt!]
\begin{centering}
\includegraphics[width=0.5\textwidth]{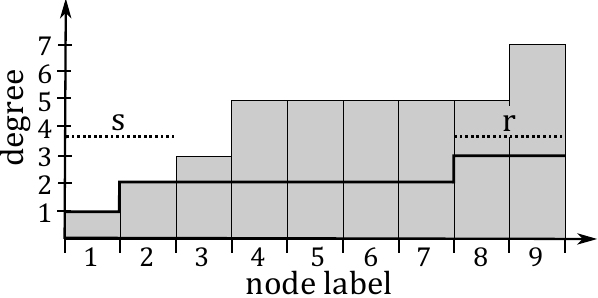} 
\caption{An illustration of the lower bound of (\ref{eq:wf.bound}).
Gray areas indicate the degrees of the nodes $(d_i)$ while the solid line
shows an optimal solution $(p_i)$. Here $s=2$ and $r=2$.\label{fig:wf.bound}}
\end{centering}
\end{figure}

It is easy to find $P$ exactly, giving an analytic lower bound on NANIP with decreasing convex $f$.
Namely, $P$ occurs when the first $s$ of the nodes are installed with $p_i=d_i$ and the rest of the nodes
are installed with $p_i$ equal to $d_s$ or $d_s+1$, as follows.
\begin{thm}\label{thm:wf.bound}
For any $G$ and decreasing convex  $f$ and any $\sigma$ we have the following bound. 
If $d_1 > m/n$, let $s=0$, otherwise let
\begin{equation}
\label{eq:def.s}
s = \max_{k \in \{1,2,\dots, n\}} \left\{ k : (n-k)d_k + \sum_{i=1}^k d_i \leq m\right\}\,.
\end{equation}
When $s=0$, any optimal solution of (\ref{eq:relaxationP}) has $n -(m - n \lfloor \frac{m}{n} \rfloor)$ elements in $(p_i)_{i=1}^n$ equal to $\lfloor \frac{m}{n} \rfloor$, 
and $m - n \lfloor \frac{m}{n} \rfloor$ elements equal to $\lfloor \frac{m}{n} \rfloor +1$;
When $s\geq 1$, let $r = m - (n-s)d_s - \sum_{i=1}^s d_i$, and an optimal solution is given by the following assignment: 
(i) $p_i = d_i$ for $i=1,2,\dots,s$; (ii) for $s<i\leq n$: $p_i = d_s+1$ for exactly $r$ indices and $p_i = d_s$ for exactly $n-s-r$ indices 
(illustrated in Fig.~\ref{fig:wf.bound}). 
Thus the cost of NANIP can be bounded from below by 
\begin{equation}
\label{eq:wf.bound}
C_G(\sigma) \geq P=\begin{cases}
\left(n -(m - n \lfloor \frac{m}{n} \rfloor)\right)f(\lfloor \frac{m}{n} \rfloor) +  \left(m - n \lfloor \frac{m}{n} \rfloor\right)f(\lfloor \frac{m}{n} \rfloor + 1)\,, &\mbox{when } s=0 \,,\\
 (n-r-s)f(d_s) + r f(d_s + 1) + \sum_{i=1}^s f(d_i)\,,  &\mbox{when } s\geq 1\,.
\end{cases}
\end{equation}
\end{thm}

\begin{proof}
Convexity implies that for any integers $a<b$ in $\mathbb{N}_0$, 
\begin{equation}
f(a)+f(b) \geq f(a+1)+f(b-1)
\label{eq:cvx2}\,.
\end{equation}
where the equality holds if and only if $b-a = 1$. 
In the case where $d_1 \leq m/n$, (\ref{eq:def.s}) is well-defined and $s\geq 1$ 
because $n d_1 \leq m$ and $\sum_{i=1}^n d_i = 2m$. 
Suppose that there is an optimal assignment $(q_i)_{i=1}^n \neq (p_i)_{i=1}^n$.
If there is $j \leq s$ such that $q_j < d_j$ then from (\ref{eq:def.s}) must exist $j' > s$ such that $q_j \leq q_{j'} - 2$. 
From (\ref{eq:cvx2}) it follows that by assigning $q_j : = q_j + 1$ and $q_{j'}:=q_{j'} - 1$ 
we can improve the existing solution (or at worst, keep its cost unchanged).
Hence, $q_j = d_j$ for all $j=1,2,\dots,s$. 
Suppose now that there exists a pair of indexes $s < k < k' \leq n$ such that $| q_k - q_{k'} | \geq 2$.
Applying the same transformation on $q_k$ and $q_{k'}$ as in the previous case we can match or improve upon the existing solution. 
Hence for all $s < \ell < \ell' \leq n$ we have $|q_{\ell} - q_{\ell'}| = 1$. 
Finally, from the invariant $\sum_{i=1}^n q_i = m$ it follows that among the numbers $q_{s+1}, q_{s+2}, \dots, q_n$ 
there are $r$ of them equal to $d_s$ and $n-s-r$ of them equal to $d_s +1$, which completes the proof.
The case $s = 0$ is simpler. From $s=0$ it follows that $\lfloor \frac{m}{n} \rfloor \leq d_1 \leq d_2 \leq \cdots \leq d_n$. 
Given a solution $(p_i)_{i=1}^n$, if there were $i,j$ such that $|p_i - p_j| \geq 2$ then we could match or improve $(p_i)_{i=1}^n$ by transformations similar to above. 
\end{proof}
As a corollary, one could show that the same bound applies in the case of non-strictly convex $f$, which might arise in applications.
Given such an $f$, infinitesimally perturb it to obtain $g$ which is strictly convex, 
where $g(d) = f(d) + \epsilon^d$, for some sufficiently small $\epsilon>0$ and $d \in \{0, 1, \dots, n-1\}$.
The cost $C_G(\sigma)$ with $g$ is bounded by (\ref{eq:wf.bound}) and must be greater than the cost $C_G(\sigma)$ with $f$,
but the difference vanishes in the limit $\epsilon\to 0$.

\section{Computation}\label{sec:computation}
We now consider the hardness of solving NANIP and introduce three algorithms for the problem, an Integer Programming Formulation, a Dynamic Programming algorithm, and a fast heuristic. 

\begin{thm}{The Neighbor Aided Network Installation Problem is NP-hard.}
\end{thm}

\begin{proof}
NP-hardness can be established by reduction from the Max Independent Set problem. Given a graph $G$, we can find a maximal independent set by seeking the optimal solution to the instance of our problem with the same $G$ and the function $f$ defined by $f(0)=0$ and $f(k)=1$ for all $k \not = 0$. Using that $f$, any optimal permutation found by solving our problem will necessarily have the maximum possible number of nodes with cost 0, and the rest with cost 1. The nodes with cost 0 then necessarily form a (maximal) independent set, since if any two were adjacent then the latter of the two to be installed would have incurred a cost of $f(1)=1$.
\end{proof}
Note that while this establishes that general NANIP is NP-hard, it does not apply to the case of a decreasing convex cost function,
which remains an open problem.

\subsection{Integer Programming Formulation\label{sub:IP}}
NANIP with a decreasing convex cost $f$ can be formulated as an integer program in the following way.
For all $i \in V$, let $X_{it}=1$ iff $\sigma(t)=i$, that is, $X_{it} \in \{0,1\}$ denotes that node $i$ is installed at time $t \in \{1,\dots,n\}$.
Observe that if and only if $i$ is installed before $j$ then for some $T\in \{1,2,\dots,n-1\}$, $\sum_{t=1}^T{X_{it}} - \sum_{t=1}^T{X_{jt}} = 1$. 
For each edge $(i,j)$ (considering direction) introduce a weight $E_{ij} \in [0,1]$ so that $E_{ij} + E_{ji} = 1$ for all $i,j\in V$.  
We will ensure that at all optimal solutions, $E_{ij}$ takes integer values.
Through $E_{ij}$ variables we express $r(j,G,\sigma)$, the number of neighbors of a node $j$ installed before $j$: $r(j,G,\sigma) = \sum_{i\in N(j)}{E_{ij}}$.
Lastly, let $l_d(x)$ define the downward slopping line through the points $(d-1,f(d-1))$ and $(d,f(d))$ with $d\in \{1,2,\dots,\}$. 
Collectively those lines define a polyhedron whose extreme points coincide with the values of $f(d)$.
Thus $l_d(x) = f(d) + \left(f(d) - f(d-1)\right)(x-d)$.

NANIP is the problem
\begin{align}
\min &\sum_{i=1}^{n} c_i.\label{eq:NANIP-IP}& \\
\mbox {s.t. ~~~} & c_j \geq f(d) + \left(f(d) - f(d-1)\right)\left(\sum_{i\in N(j)}{E_{ij}}-d\right) & j\in V, d=1,2,\dots,n-1 \label{eq:NANIP-ld}\\
 & E_{ij} \geq \sum_{t=1}^T\left(X_{jt}-X_{it}\right) & T=1,2,\dots,n-1, i\in V, j \in N(i)\notag\\
 & E_{ij} + E_{ji} = 1 & i\in V, j \in N(i) \notag\\
 & \sum_{t=1}^n X_{jt} = 1, \mbox{ and }  \sum_{i=1}^n X_{it} = 1 & j\in V, t=1,2,\dots,n \notag
\end{align}

Observe that because of the convexity of $f$, the points in the intersection of the inequalities $(\ref{eq:NANIP-ld})$ 
are exactly the points above the convex hull of $\{f(d)\}_{d=1}^n$.
At an optimal $\sigma$, the value of $c_i$ would minimize $f(x)$ subject to $x\leq\sum_{i\in N(j)}{E_{ij}}$, so that
the value of $c_i = f(r(i,G,\sigma))$ .

Using CPLEX version 12.2 (IBM Corp.) we were able to solve NANIP on tree graphs with $50$ nodes in less than one hour, 
and with longer running time as the number of edges increases. 
Convergence is greatly accelerated when the basic NANIP is generalized to allow for realistic variation in the cost of a node ($f(d)$ becomes $f_i(d)$ for each 
$i \in V$).  
The speedup is likely due to the removal of symmetry in the problem, which is very high.
For instance, in any graph on $2$ or more nodes and any $\sigma$, swapping the first and second nodes in $\sigma$ does not change its cost.

\subsection{Dynamic Programming Algorithm\label{sub:DP}}
Consider a generalization of NANIP which we call the Subgraph-Aided Network Installation Problem (SANIP).
SANIP allows the cost  $r(v_t, G, \sigma)$ of a node at time $t$ to be an arbitrary 
function of the subgraph induced by the first $t-1$ nodes of $\sigma$, rather than just the neighbors of the node. 
This cost may even depend upon the installation time because it is the number of nodes $|V(H)|$. 

SANIP instances can be solved by the following Dynamic Programming Algorithm, Algorithm~\ref{al:DP}, 
which exploits the fact that the cost of any node $u$
depends only upon $H$ and not how $H$ was installed.
Instead of considering the space of permutations on $n$ nodes, which has size $n!$, it needs to consider
merely the set of subsets of $n$ nodes, which grows as $2^n$. 
Concretely, to find an optimal installation of a subgraph on nodes $T\subset V$,
it considers every decomposition of the form $T=S\cup\left\{ u\right\}$.
The algorithm maintains a table $TOPT$ indexed by subsets $S\subset V$, 
and stores a least-cost sequence for installing $S$ along with its cost.
The maximum size of the $TOPT$ table is ${n \choose n/2}$.
In the pseudocode, the symbol $\vee$ on sequences denotes concatenation.

\begin{algorithm}[ht]\caption{Dynamic Programming Algorithm for Construction of $G$}     
\label{al:DP} 
\begin{algorithmic}
\STATE $TOPT[\{u\}]\leftarrow (u)$ for all $u\in V$.
\FORALL{$t \in \{2\dots n$\}}
    \STATE $TOPT' \leftarrow \varnothing$
    \FORALL{$S \in TOPT$}
        \FORALL{$u \in V\setminus S$}
            \STATE $T \leftarrow S\cup\left\{ u\right\}$
            \IF {$T\notin TOPT'$ {\bf or} $cost(TOPT[S])+c(u,S) < cost(TOPT'[T])$}
                \STATE $TOPT'[T] \leftarrow TOPT[S] \vee (u)$
            \ENDIF
        \ENDFOR 
    \ENDFOR 
    \STATE free $TOPT$
    \STATE $TOPT \leftarrow TOPT'$
\ENDFOR 
\STATE \textbf{Output} $TOPT[V]$
\end{algorithmic} 
\end{algorithm}

The algorithm uses $\Theta(2^{n/2})$ space (the table size peaks at iteration $n/2$ and then declines) and $\Theta(n2^{n})$
time. We found that on a desktop computer the computation is tractable up to $n\approx25$. 
While the space requirement might be improved, the running time cannot be improved beyond $\Omega(n2^{n})$.
The lower bound occurs because it is possible to have an instance where the cost of installing some node $u$ 
is arbitrarily small if $u$ is sequenced immediately after some subgraph $H$ of $G$. 
Observe also that the specification of SANIP would in general use $\Theta(n2^{n})$ space.
To solve instances of larger size this algorithm could be replaced by a dynamical programming algorithm with a finite horizon.

\subsection{A Linear-Time Heuristic\label{sub:Greedy}}
Exponential time algorithms can be avoided if suboptimal solutions are acceptable, which is often the case in practice.
We will show that a simple greedy algorithm runs in linear time and we will show results of simulations comparing the results of the greedy algorithm with those of Algorithm~\ref{al:DP}.

Algorithm~\ref{al:greedy} below preferentially installs nodes that have the lowest installation cost,
which is to say, nodes that have the most installed neighbors.
The cost of the solution found by this heuristic can be bounded when $f$ is decreasing convex, as follows. In the worst case, the first node has cost $f(0)$, and all other nodes are constructed at cost $f(1)$. Often the bound can be made tighter, as follows.

\begin{algorithm}[ht]\caption{Greedy Algorithm for NANIP}     
\label{al:greedy} 
\begin{algorithmic} 
\STATE $\sigma_{\textrm{aux}}\leftarrow\varnothing$ 
\STATE \textbf{If} {n=0} \textbf{Output} $\sigma_{\textrm{aux}}$ 
\STATE $\sigma_{\textrm{aux}}\leftarrow $ Random $u \in V$ 
\WHILE{$|\sigma_{\textrm{aux}}| < n$}     
\STATE $u\in \argmax_{u\in V\setminus \sigma_{\textrm{aux}}} \left\{f(r(u, G, \sigma_{\textrm{aux}})) : u\in V\setminus \sigma_{\textrm{aux}}\right\}$, resolving ties arbitrarily.
\STATE $\sigma_{\textrm{aux}}\leftarrow \sigma_{\textrm{aux}} \vee u$.
\ENDWHILE 
\STATE \textbf{Output} $\sigma_{\textrm{aux}}$
\end{algorithmic} 
\end{algorithm}

\begin{prop}
Let $f$ be decreasing convex and $G$ a graph with node degrees $d_{1}\le d_{2}\leq\dots\le d_{n}$ and $n\geq 3$. The cost of the solution $\sigma$ found by Algorithm~\ref{al:greedy} is bounded as
\label{thm:breadth}
\begin{equation}
C_G(\sigma) \leq f(0)+s f(1)+f(q)+f(d_{s+3})+\dots+f(d_n)\,,\label{eq:Breadth-bd}
\end{equation}
where $s\in \{0,1,\dots,n-1\}$ is the largest possible, subject to the constraints $(i)$ $q\in \{1,2,\dots,d_{s+2}\}$ and $(ii)$ $s + q + d_{s+3}+\dots+d_n = m$. 
\end{prop}
\noindent 
In (\ref{eq:Breadth-bd}), the terms after $s f(1)$ occur only if $d_n > 1$.

The proof, using techniques similar to the proof of Theorem~\ref{thm:wf.bound}, is elementary. 
Given that we can bound from below the cost of the optimal solution (Corollary~\ref{cor:jensen}), one can even compute an optimality gap for any instance of NANIP.

The practical performance of Algorithm~\ref{al:greedy} is illustrated by simulations in Fig.~\ref{fig:performance}.
Algorithm~\ref{al:greedy} is within $5\%$ of the optimum in our simulations in networks of low, medium and high edge densities.
Therefore, and because it runs in linear time, Algorithm~\ref{al:greedy} might be appropriate to large-scale practical instances of NANIP.
Also shown in Fig.~\ref{fig:performance} are the degree upper bound and the relaxation bound.
The simulations take a random connected graph on $15$ nodes and $m$ edges and
the convex function $f(i)=\frac{1}{1+i}$ for $i\in\mathbb{N}_0$ (the results are qualitatively similar for dozens of convex cost functions we examined.)
For each value of $m$, $5$ random connected graphs were generated by taking a random tree on $n$ nodes and adding $m-(n-1)$ edges to it at random.
The greedy algorithm was scored based on the average over $10$ runs on each graph.

\begin{figure}[tbh]
\begin{centering}
\includegraphics[width=0.6\textwidth]{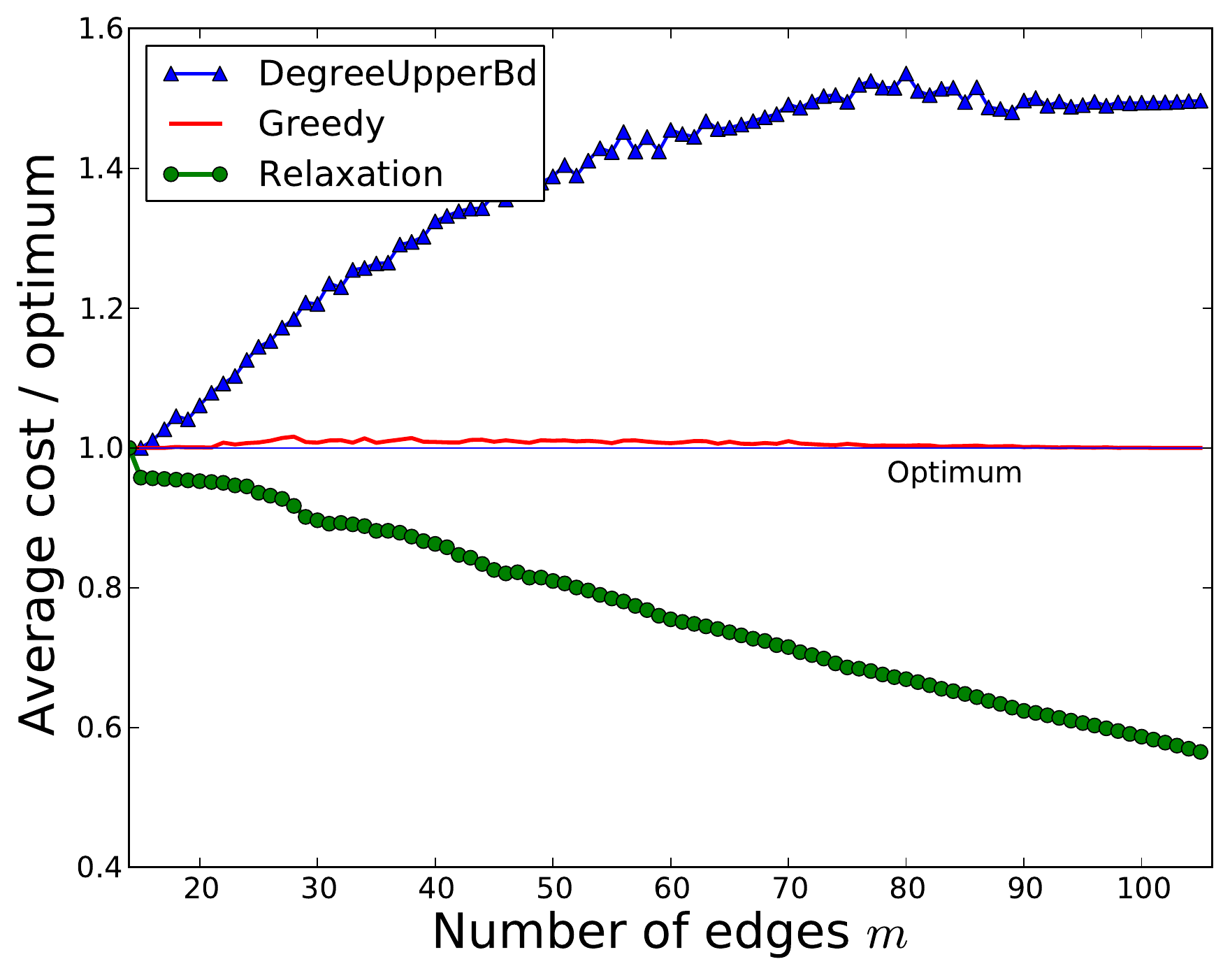} 
\par
\end{centering}
\caption{The performance of bounds and of Algorithm~\ref{al:greedy} (Greedy) as a function of the number of edges, $m$.
The greedy heuristic usually finds solutions indistinguishable from the optimum.
The degree upper bound is from (\ref{eq:Breadth-bd}) and the relaxation bound is from (\ref{eq:wf.bound}).
\label{fig:performance}}
\end{figure}

In addition to this greedy heuristic, we have examined a heuristic which installs nodes based on their degree
starting with the highest degree nodes.
In our simulations, this heuristic did not outperform the cost-greedy heuristic (Alg.~\ref{al:greedy})
even when the cost function $f$ was decreasing concave (detailed results are not shown).
Thus, Alg.~\ref{al:greedy} seems to be appropriate for the case of a decreasing concave $f$.

\section{Conclusion} \label{sec:concl}
This paper introduces a new discrete optimization problem, the Neighbor-Aided Network Installation Problem (NANIP).
We found that NANIP and its generalization could be solved using integer and dynamic programming,
while a fast greedy approach exists for the case of convex decreasing $f$.
We therefore find that recovery operations of simple infrastructure networks
could be planned by a simple rule: every step of the recovery should focus on the most accessible of the damaged network nodes.

\section*{Acknowledgments}
We thank Constantine Caramanis, Leonid Gurvits, Jason Johnson, Joel Lewis and Joel Nishimura for insightful conversations. 
AG was funded by the Department of Energy at the Los Alamos National Laboratory
under contract DE-AC52-06NA25396 through the Laboratory Directed Research
and Development program, and by the Defense Threat Reduction Agency
grant {}``Robust Network Interdiction Under Uncertainty''. 
MB is supported in part by NIST grant 60NANB10D128. Part of this work was done at Los Alamos.
We thank Feng Pan and Aric Hagberg for the encouragement,
and two anonymous reviewers for their constructive criticism.

\bibliographystyle{plain}
\bibliography{../setseq}

\begin{thebibliography}{10}

\bibitem{Guha99}
Sudipto Guha, Anna Moss, Joseph~(Seffi) Naor, and Baruch Schieber.
\newblock Efficient recovery from power outage (extended abstract).
\newblock In {\em Proceedings of the thirty-first annual ACM symposium on
  Theory of computing}, STOC '99, pages 574--582, New York, NY, USA, 1999. ACM.

\bibitem{Karp61}
Michael Held and Richard~M. Karp.
\newblock A dynamic programming approach to sequencing problems.
\newblock In {\em ACM '61: Proceedings of the 1961 16th ACM national meeting},
  pages 71.201--71.204, New York, NY, USA, 1961. ACM.

\bibitem{Lee07}
E.E. Lee, J.E. Mitchell, and W.A. Wallace.
\newblock Restoration of services in interdependent infrastructure systems: A
  network flows approach.
\newblock {\em Systems, Man, and Cybernetics, Part C: Applications and Reviews,
  IEEE Transactions on}, 37(6):1303 --1317, Nov 2007.

\bibitem{Mitchell96}
J.E. Mitchell and B.~Borchers.
\newblock Solving real-world linear ordering problems using a primal-dual
  interior point cutting plane method.
\newblock {\em Annals of Operations Research}, 62(1):253--276, 1996.

\bibitem{Erdos60}
P.~Erd\H os and A.~R\'enyi.
\newblock {On the evolution of random graphs}.
\newblock {\em Publication of the Mathematical Institute of the Hungarian
  Acadamy of Science}, 5:17--67, 1960.

\bibitem{Pardalos94}
P.M. Pardalos, F.~Rendl, and H.~Wolkowicz.
\newblock The quadratic assignment problem: A survey and recent developments.
\newblock Technical Report CORR 94-06, University of Waterloo, 1994.

\bibitem{schrijver2005history}
A.~Schrijver.
\newblock {On the history of combinatorial optimization (till 1960)}.
\newblock {\em Handbooks in Operations Research and Management Science},
  12:1--68, 2005.

\bibitem{Hentenryck10}
P.~Van~Hentenryck, R.~Bent, and C.~Coffrin.
\newblock Strategic planning for disaster recovery with stochastic last mile
  distribution.
\newblock In Andrea Lodi, Michela Milano, and Paolo Toth, editors, {\em
  Integration of AI and OR Techniques in Constraint Programming for
  Combinatorial Optimization Problems}, volume 6140 of {\em LNCS}, pages
  318--333. Springer Berlin / Heidelberg, 2010.

\bibitem{Mathematica}
{Wolfram Research Inc.}
\newblock Mathematica 6.0, May 2007.
\newblock Champaign, Illinois.

\bibitem{Yu03}
Gang Yu, Michael Arguello, Gao Song, Sandra~M. McCowan, and Anna White.
\newblock {A New Era for Crew Recovery at Continental Airlines}.
\newblock {\em Interfaces}, 33(1):5--22, 2003.

\end{thebibliography}

\end{document}